\documentclass[12pt,reqno]{amsart}

\usepackage{amssymb}
\usepackage{amscd}
\usepackage{graphicx}
\usepackage{amsmath}
\usepackage{setspace}
\usepackage[all]{xy}
\usepackage{color}
\usepackage{enumerate}

\newtheorem{theorem}{Theorem}[section]

\newtheorem{corollary}[theorem]{Corollary}

\newtheorem{lemma}[theorem]{Lemma}

\newtheorem{proposition}[theorem]{Proposition}

\theoremstyle{remark}

\newtheorem{example}[theorem]{Example}
\newtheorem{remark}[theorem]{Remark}

\newcommand{\CH}{\mathcal{H}}

\newcommand{\Z}{{\mathbb Z}}

\newcommand{\Ima}{\mathrm{Im}}

\newcommand{\End}{\mathrm{End}}

\newcommand{\Heap}{\mathrm{ Heap}}

\def\lto{\longmapsto}
\def\lra{\longrightarrow}
\newcounter{rlist} 
   
  \def\eps{\varepsilon}

\title{The Baer-Kaplansky Theorem for All Abelian Groups and Modules}

\author{Simion Breaz}

\address{Babe\c s-Bolyai University, Faculty of Mathematics and Computer Science, Str. Mihail Kog\u alniceanu 1, 400084, Cluj-Napoca, Romania}

\email{bodo@math.ubbcluj.ro}

\author{Tomasz Brzezi\'nski}
\thanks{This research  of T.\ Brzezi\'nski is partially supported by the National Science Centre, Poland, grant no. 2019/35/B/ST1/01115.}
\address{
Department of Mathematics, Swansea University, 
Swansea University Bay Campus,
Fabian Way,
Swansea,
  Swansea SA1 8EN, U.K.\ \newline \indent
Faculty of Mathematics, University of Bia{\l}ystok, K.\ Cio{\l}kowskiego  1M,
15-245 Bia\-{\l}ys\-tok, Poland}

\email{T.Brzezinski@swansea.ac.uk}

\begin{document}

\begin{abstract}
It is shown that the Baer-Kaplansky theorem can be extended to all abelian groups provided that the rings of endomorphisms of groups are replaced by trusses of endomorphisms of corresponding heaps. That is, every abelian group is determined up to isomorphism by its endomorphism truss and every isomorphism between two endomorphism trusses associated to some abelian groups $G$ and $H$ is induced by an isomorphism between $G$ and $H$ and an element from $H$. This correspondence is then extended to all modules over a ring by considering heaps of modules. It is proved that the truss of endomorphisms of a heap associated to a module $M$ determines $M$ as a module over its endomorphism ring.   
\end{abstract}


\maketitle

\section{Introduction}

The Baer-Kaplansky Theorem, \cite[Theorem 16.2.5]{Fu15}, states that, for every isomorphism $\Phi:\End(G)\lra \End(H)$ between the endomorphism rings of torsion abelian groups $G$ and $H$ there exists an isomorphism $\varphi:G\lra H$ such that $\Phi:\alpha\lto \varphi\alpha \varphi^{-1}$ (here and throughout this note the composition of mappings is denoted by juxtaposition of symbols). A similar result was known for endomorphisms of vector spaces. The existence of an analogous statement for other kinds of modules (or abelian groups) was investigated by many authors, e.g.\ in \cite{Iv}, \cite{May}, \cite{Wi}, and \cite{Wo}. There are situations in which there exist isomorphisms between endomorphism rings that are not induced in the above mentioned way. Such an example is described  in \cite[p.\ 486]{May}, and some other detailed studies are presented in \cite{Iv} and \cite{Wo}. On the other hand, there are examples which show that, in general, in order to obtain a Baer-Kaplansky Theorem, one needs to restrict to some reasonable classes of objects, see e.g.\ \cite[Example 9.2.3]{Fu15}. From a different perspective, it was proven in \cite{Bre15} that in the case of modules over principal ideal domains every module is determined up to isomorphism by the endomorphism ring of a convenient module. In particular, two abelian groups $G$ and $H$ are isomorphic if and only if the rings $\End(\Z\oplus G)$ and $\End(\Z\oplus H)$ are isomorphic. However, if $G$ contains a direct summand isomorphic to $\Z$ then not every ring isomorphism between $\End(\Z\oplus G)$ and $\End(\Z\oplus H)$ is induced in a natural way by an isomorphism between $G$ and $H$.    

A natural question that emerges from this discussion is this: can one associate some endomorphism structures to abelian groups  such that all isomorphisms between two such endomorphism structures are induced in a natural way by isomorphisms between the corresponding abelian groups? 

In this note we present an answer to this question that makes use of an observation made by Pr\"ufer  \cite{Pr} that the structure of an abelian group $(G,+)$ can be fully encoded by a set $G$ with a ternary operation modelled on $[a,b,c]=a-b+c$ 
(this observation was extended to general groups by Baer \cite{Ba}). A set $G$ together with such a ternary operation is called a \textit{heap}, and a heap is abelian if $[a,b,c]=[c,b,a]$, for all $a,b,c\in G$; the precise definition is recalled in the following section. By fixing the middle entry in the ternary operation, a (abelian) group can be associated to a (abelian) heap uniquely up to isomorphism.  In this way we obtain a bijective correspondence between the isomorphism classes of (abelian) groups and the isomorphism classes of (abelian) heaps. Details about these correspondences are given e.g.\ in \cite{Brz20}, \cite{Ce}, and \cite{La}. The set $E(G)$ of all endomorphisms of an abelian heap is a heap with the ternary operation defined pointwise. Being an endomorphism object it has a natural monoid structure given by composition. The composition distributes over the ternary heap operation thus making $E(G)$ a (unital) \textit{truss}, a notion proposed  in \cite{Brz19} as an algebraic structure encapsulating both rings and \textit{(skew) braces} introduced in \cite{Ru}, \cite{CedJes:bra}, \cite{GuaVen:ske} to capture the nature of solutions to the set theoretic Yang-Baxter equation \cite{Dr}. 

The endomorphism truss $E(G)$ of an abelian group $G$ carries more information than the endomorphism ring $\End(G)$ (just as the holomorph of a group carries more information than the group of its automorphisms). For example $E(G)$ includes all constant mappings. In fact  $E(G)$ can be realised as a semi-direct product $G\rtimes \End(G)$, see \cite[Proposition~3.44]{Brz20}. It seems quite reasonable to expect that  $E(G)$ provides a right environment for the Baer-Kaplansky theorem for all abelian groups. In this note we show that this is indeed the case, and the first main result, Theorem~\ref{main-BK}, establishes the Baer-Kaplansky type correspondence between the isomorphisms of abelian groups $(G,+)$ and $(H,+)$, and isomorphisms of their endomorphism trusses $E(G)$ and $E(H)$. In the second main result, Theorem~\ref{main-BK-mod}  we extend the Baer-Kaplansky correspondence to all modules $M$ over a ring $R$. To achieve this we associate to each module over a ring a family of modules, which we term the heap of modules and we define homomorphisms of heaps of modules as maps that respect all these module structures in a specific way. Endomorphisms of heaps of modules form trusses, and these trusses for two modules (over possibly different rings) are isomorphic  if and only if the modules are equivalent as modules over their own endomorphism rings. The equivalence of modules over different rings is defined in a natural way as an isomorphism in the category consisting of pairs of rings and (left) modules over these rings $(R,M)$, with morphisms from $(R,M)$ to $(S,N)$ given as pairs consisting of a ring homomorphism $\varrho: R\lra S$ and a homomorphism of modules $\mu:M\lra N$ over $R$, with $N$ viewed as the $R$-module via $\varrho$.

\section{Morphisms between endomorphism trusses}

A \textsl{heap} is a set $G$ together with a ternary operation $[-,-,-]$ on $G$ that is associative and satisfies the Mal'cev identities, that is,  for all $a,b,c,d,e\in G$, 
$$
[[a,b,c],d,e]=[a,b,[c,d,e]] \quad \mbox{and} \quad 
[a,a,b]=b=[b,a,a].
$$
The heap $(G,[-,-,-])$ is \textit{abelian} if $[a,b,c]=[c,b,a]$ for all $a,b,c\in G$.

A heap structure on a set $G$ induces a group structure on $G$, and this group structure is unique up to isomorphism. More precisely, if $(G,[-,-,-])$ is a heap and $b\in G$ then the binary operation $+_b$ defined by 
$$
a+_b c=[a,b,c], \quad \textrm{ for all }a,b,c\in G,
$$ 
equips $G$ with a group structure. The group $(G,+_b)$ is referred to as a \textit{retract} of $(G,[-,-,-])$  at $b$. The operation $+_b$ is commutative provided $G$ is an abelian heap. Moreover, if $b'\in G$ then the groups $(G,+_b)$ and $(G,+_{b'})$ are isomorphic by the map $a\lto [a,b,b']$.

Conversely, if $(G,+)$ is a group then $G$ together with the ternary operation
$$
[a,b,c]=a-b+c, \quad \textrm{ for all }a,b,c\in G,
$$
is a heap, which is abelian if and only if $(G,+)$ is abelian. In this case, for every $b\in G$, the induced group structure $(G,+_b)$ is isomorphic to $(G,+)$ since $+=+_0$. Unlike the assignment of a retract to a heap, this assignment of the heap to a group is functorial, that is, it defines a functor from the category of (abelian) groups to that of (abelian) heaps.

If $(G,[-,-,-])$ and $(H,[-,-,-])$ are heaps, a mapping $\alpha:G\lra H$ is a \textit{heap morphism} if $\alpha([a,b,c])=[\alpha(a),\alpha(b),\alpha(c)]$, for all $a,b,c\in G$. The set of all heap morphisms from $G$ to $H$ is denoted by $\Heap(G,H)$. If $G$ and $H$ are abelian heaps, then $\Heap(G,H)$ is a heap with the pointwise defined heap operation $[\alpha,\beta,\gamma](a)=[\alpha(a),\beta(a),\gamma(a)]$. We note that $\alpha:G\lra H$ is a heap morphism if and only if $\alpha :(G,+_b)\lra (H,+_{\alpha(b)})$ is a group morphism for every $b\in G$. If $G$ and $H$ are groups, a map $\alpha:G\lra H$ is a {heap morphism} if it is a morphism between the induced heap structures on $G$ and $H$. Every group morphism is a heap morphism, but this property is not reciprocated.  In a standard way, a heap morphism with equal domain and codomain is called a \textit{heap endomorphism}. The set of all heap endomorphisms of an abelian group $G$ is denoted by $E(G)$. Clearly, $E(G)$ is a monoid with the composition of maps as a multiplicative operation. However,  the multiplicative operation (composition) does not distribute over the pointwise defined addition. On the other hand, it does distribute over the pointwise defined ternary heap operation $[\alpha,\beta,\gamma](a)=\alpha(a)-\beta(a)+ \gamma(a)$,  that is,
$$
\delta[\alpha,\beta,\gamma]=[\delta\alpha,\delta\beta,\delta\gamma] \quad \mbox{and}\quad  [\alpha,\beta,\gamma]\delta=[\alpha\delta,\beta\delta,\gamma\delta], 
$$
for all $\alpha,\beta,\gamma, \delta\in E(G)$. An abelian heap together with a semigroup operation that distributes over the heap operation in the above sense  is called a \textit{truss}. A morphism of trusses is a mapping that is both a homomorphism of heaps and semigroups. Thus the algebraic system $(E(G), [-,-,-],\cdot)$ is a {truss}. 

If $a\in G$, then we will denote by $\widehat{a}$ the constant map 
$$
\widehat{a}: G\lra G, \qquad b\lto a.
$$
Thanks to the fact that the heap operation is idempotent, every $\widehat{a}$ is an element on $E(G)$. Clearly, the set $\widehat{G}$ of all constant maps is closed under the heap operation and the composition, hence it forms a sub-truss of $E(G)$. The map $\widehat{\_}: G\lra \widehat{G}$ is an isomorphism of heaps. The multiplication in $E(G)$ transferred back to $G$ equips the heap $(G,[-,-,-])$ with the lopsided truss structure $ab = a$, for all $a,b\in G$ (of no importance in this note, though).

The proof of the following characterization for heap morphisms is a simple exercise.

\begin{lemma}\label{heap-morphism}
Let $(G,+)$ and $(H,+)$ be abelian groups. 
\begin{itemize}
\item[(i)] A map $\varphi:G\lra H$ is a heap morphism (resp.\ isomorphism) if and only if there exists a group morphism (resp.\ isomorphism) $\widetilde\varphi:G\lra H$ and an element $h_0\in H$ such that $\varphi(\_)=\widetilde\varphi(\_) +h_0$. In this case, 
$\widetilde\varphi(\_)=\varphi(\_)-\varphi(0)$ and $h_0=\varphi(0)$ are uniquely determined by $\varphi$. 
\item[(ii)] A heap morphism $\varphi:G\lra H$ is constant  if and only if $\varphi \hat{a} =\varphi$, for all $\hat{a}\in \widehat{G}$. In this case $\varphi \alpha =\varphi$, for all $\alpha\in E(G)$.
\item[(iii)] For all $a\in G$, $\varphi\widehat{a} = \widehat{\varphi(a)}$.
\end{itemize}
\end{lemma}

The first main result of this note is contained in the following theorem.

\begin{theorem}[The Baer-Kaplansky theorem for abelian groups]\label{main-BK}
Two abelian groups $(G,+)$ and $(H,+)$ are isomorphic if and only if their endomorphism trusses $E(G)$ and $E(H)$ are isomorphic. 

Furthermore, for every isomorphism  $\Phi:E(G)\lra E(H)$ of trusses there exists a unique heap isomorphism $\varphi:G\lra H$ such that $\Phi(\alpha)=\varphi \alpha \varphi^{-1}$ for all $\alpha\in E(G)$.    
\end{theorem}

\begin{proof}
For all $a\in G$ and $b\in H$, 
$$
\Phi(\widehat{a})\widehat{b}=\Phi(\widehat{a}\Phi^{-1}(\widehat{b}))=\Phi(\widehat{a}).
$$
It follows that $\Phi(\widehat{a})\in \widehat{H}$. 
Therefore, one can define functions 
 $\varphi:G\lra H$ 
 and $\psi:H\lra G$, by setting
 $$
 \widehat{\varphi(a)} = \Phi(\widehat{a}), \qquad \widehat{\psi(b)} = \Phi^{-1}(\widehat{b}),
 $$
 for all $a\in G$ and $b\in H$. Note that, since $ \Phi(\widehat{a})$ is a constant function, $\varphi(a) = \Phi(\widehat{a})(b)$, for all $b\in H$ (in particular for $b=0$).
Since both $\Phi$ and $\Phi^{-1}$ are heap morphisms, $\varphi$ and $\psi$ are heap morphisms too. Moreover, $\varphi\psi=1_H$ since
$$
\widehat{\varphi\psi(b)}= \Phi(\widehat{\psi(b)}) = \Phi(\Phi^{-1}(\widehat{b}))=\widehat{b},
$$
for all $b\in H$. In an analogous way one proves that $\psi\varphi=1_G$, hence $
\psi = \varphi^{-1}$ and $\varphi$ is an isomorphism of heaps. In view of  Lemma~\ref{heap-morphism} the required isomorphism of abelian groups is thus obtained as $a\lto \varphi(a) - \varphi(0)$.

The assignment $\Phi\lto \varphi$ just constructed can be understood as the maping $\Theta$,
defined by 
$$
(\Phi:E(G)\lra E(H))\overset{\Theta}\lto (\Phi(\widehat{\_})(0):G\lra H)\, ,
$$ 
from the set of truss isomorphisms $E(G)\lra E(H)$ to the set of heap isomorphisms $G\lra H$. The inverse map $\Upsilon$ is defined by 
$\Upsilon(\varphi)(\alpha)=\varphi \alpha \varphi^{-1}$, for all heap isomorphisms $\varphi:G\lra H$ and all $\alpha \in E(G)$. Clearly,  $\Upsilon(\varphi)$ is an isomorphism of trusses.

To check that $\Theta$ and $\Upsilon$ are mutual inverses, take any heap isomorphism
$\varphi:G\lra H$ and $a\in G$, and compute 
$$
\widehat{\Theta\Upsilon(\varphi)(a)}=(\Upsilon(\varphi))(\widehat{a})=\varphi \widehat{a}\varphi^{-1}=\varphi\widehat{a}=\widehat{\varphi(a)},
$$
where the last two equalities follow by Lemma~\ref{heap-morphism}.
Therefore, $\Theta\Upsilon(\varphi)=\varphi$,  for all heap isomorphisms $\varphi:G\lra H$. 

In the converse direction, for all truss isomorphisms $\Phi:E(G)\lra E(H)$, and $\alpha\in E(G)$, both 
$\Theta(\Phi)\alpha$ and $\Phi(\alpha)\Theta(\Phi)$ are heap morphisms from $G$ onto $H$. Moreover, for all $a\in G$,
\begin{align*} 
\widehat{\Theta(\Phi)\alpha(a)}&=\Phi(\widehat{\alpha(a)})=\Phi(\alpha\widehat{a})
=\Phi(\alpha)\Phi(\widehat{a}) 
=\Phi(\alpha)\widehat{\Theta(\Phi)(a)}.
\end{align*}
It follows that $\Theta(\Phi)\alpha = \Phi(\alpha)\Theta(\Phi)$,  whence
$$
\Upsilon\Theta(\Phi)(\alpha)=\Theta(\Phi) \alpha\Theta(\Phi)^{-1}=\Phi(\alpha),
$$
 for all $\alpha\in E(G)$. This completes the proof.
\end{proof}

\begin{remark}
In the first part of the proof of Theorem~\ref{main-BK} it was shown that $\Phi(\widehat{G})\subseteq \widehat{H}$. Let us mention that this conclusion is valid for all surjective semigroup morphisms $(E(G),\cdot)\lra (E(H),\cdot)$. This can be also obtained by using the equality
$$
\widehat{G}=\{\alpha\in E(G)\mid \alpha\beta=\alpha \text{ for all }\beta\in E(G)\},
$$
meaning that the set of constant maps $G\lra G$ coincides with the set of left absorbing elements in the semigroup $(E(G),\cdot)$. 

Since $\widehat{G} \cong G$ as heaps and every truss isomorphism $E(G)\lra E(H)$ restricts to the truss isomorphism $\widehat{G}\lra \widehat{H}$, it induces a heap isomorphism $G\lra H$. 
\end{remark}

\begin{remark}
In view of the correspondence between abelian groups and heaps, in particular, since the transformation of a heap to any of its retracts and then back to the heap yields identity, one can reformulate Theorem~\ref{main-BK} and prove it entirely in the heap phraseology. Specifically, the mapping $\Theta$ in the proof of Theorem~\ref{main-BK}  establishes a bijective correspondence between isomorphisms of all abelian heaps and isomorphisms between their corresponding endomorphism trusses.
\end{remark}

In fact, all truss morphisms $E(G)\lra E(H)$ are inner in some sense.

\begin{proposition}\label{inner-end}
Let $(G,+)$, $(H,+)$ be abelian groups and let   $\Phi:E(G)\lra E(H)$ be a truss morphism. Set 
$$
\eps_\Phi=\Phi(\widehat{0})-\Phi(\widehat{0})(0) \quad \mbox{and}\quad  e_\Phi=\Phi(\widehat{0})(0), 
$$
and define the following subset of the heap $\Heap(G,H)$ of all heap morphisms from $G$ to $H$,
$$
\Xi_\Phi:=\{\xi\in \Heap(G,H) \mid \Phi(\alpha)\xi=\xi \alpha, \textrm{ for all } \alpha\in E(G)\}.
$$ 
\begin{itemize}
\item[(a)]  The map $\eps_\Phi$ is an idempotent endomorphism of the abelian group $(H,+)$ and $\eps_\Phi(e_\Phi)=0$.
\item[(b)]  The set $\Xi_\Phi$ is not empty.
\item[(c)] For all $\xi \in \Xi_\Phi$, 
$\xi(0)\in e_\Phi+ \Ima(\eps_\Phi)$.
\item[(d)] The set $\Xi_\Phi$  
is a sub-heap of $\Heap(G,H)$, and it is isomorphic to the sub-heap $e_\Phi+ \Ima(\eps_\Phi)$ of $H$.
\end{itemize}
\end{proposition}

\begin{proof}
(a) From $\Phi(\widehat{0})=\Phi(\widehat{0})\Phi(\widehat{0})$ it follows that 
$$
\eps_\Phi+e_\Phi=(\eps_\Phi+e_\Phi)(\eps_\Phi+e_\Phi)=\eps_\Phi \eps_\Phi+\eps_\Phi(e_\Phi)+e_\Phi.
$$
Using Lemma \ref{heap-morphism}, we obtain that $\eps_\Phi$ is idempotent and $\eps_\Phi(e_\Phi)$ =0. 

(b) Let $b\in H$ and  define 
$$\xi=
\xi_b:G\lra H, \qquad  a\lto \Phi(\widehat{a})(b).
$$ 
Then, for every $\alpha\in E(G)$,
$$
\Phi(\alpha)\xi_b(a)=\Phi(\alpha)\Phi(\widehat{a})(b)=\Phi(\alpha \widehat{a})(b)=\Phi(\widehat{\alpha(a)})(b)=\xi_b(\alpha(a)),
$$
hence $\xi_b\in \Xi_\Phi$. 

(c) For any $\xi\in \Xi_\Phi$,
$$
\xi(0)=\xi(\widehat{0}(0))=\Phi(\widehat{0})\xi(0)\in \Ima(\eps_\Phi)+e_\Phi,
$$
where the second equality follows by the definition of $\Xi_\Phi$.

(d) The first statement is obvious.  Since $\Ima(\eps_\Phi)+e_\Phi$ is a coset it is a sub-heap of $H$ by \cite[Theorem~1]{Ce}. Using the same notation as in the proof of (b), we consider the map 
$$
\vartheta: \Ima(\eps_\Phi)+e_\Phi\lra \Xi_\Phi,\qquad c\lto \xi_c.
$$ 
It is easy to see that this is a morphism of heaps, so it remains only to prove that it is bijective.

Take any $\xi \in \Xi_\Phi$. Then, for all $a\in G$, 
$$
\xi(a)=\xi(\widehat{a}(0))=\Phi(\widehat{a})\xi(0)=\xi_{\xi(0)}(a).$$
Hence  $\xi=\xi_{\xi(0)}$, and the map $\vartheta$ is surjective.

Moreover, we observe that for every $c\in \Ima(\eps_\Phi)+e_\Phi$,   
$$
\xi_c(0)=\Phi(\widehat{0})(c)=\eps_\Phi(c)+e_\Phi=c+e_\Phi,
$$ 
since $\eps_\Phi$ is an idempotent group endomorphism and $\eps_\Phi(e_\Phi)=0$. 
Therefore, if $c,d\in \Ima(\eps_\Phi)+e_\Phi$ and $c\neq d$ then $\xi_c\neq \xi_d$, and the proof is completed.
\end{proof}

\begin{corollary}
Let $(G,+)$ and $(H,+)$ be abelian groups, and let $\Phi:E(G)\lra E(H)$ be a truss morphism. Suppose that there exists $a\in G$ such that $\Phi(\widehat{a})$ is a constant morphism. Then there exists a unique morphism of heaps $\xi:G\lra H$ such that 
$$ 
\Phi(\alpha)\xi=\xi\alpha, \quad \textrm{ for all }\alpha\in E(G).
$$
\end{corollary}

\begin{proof}
Let us observe that $\Phi(\widehat{0})=\Phi(\widehat{0})\Phi(\widehat{a})$ is a constant morphism. It follows that $\eps_\Phi=0$. As a consequence of  Proposition~\ref{inner-end}, there is a unique heap morphism $\xi$ such that $\Phi(\alpha)\xi=\xi\alpha$  for all $\alpha\in E(G)$.  
\end{proof}

\section{Isomorphisms of modules over endomorphism rings}\label{sec.mod}

Next, we would like to extend the Baer-Kaplansky correspondence to modules over a ring $R$. Na\"ively one could try to establish  the correspondence between $R$-module isomorphisms and endomorphism trusses of $R$-modules. One could try to follow the strategy of Theorem~\ref{main-BK}, that is, view the additive structure of an $R$-module $M$ as a heap and then study the $R$-linear heap endomorphisms of $M$, that is, heap endomorphisms that commute with the $R$-action. There are at least two difficulties with implementing this strategy. The first and rather technical problem is that the constant heap endomorphism is not $R$-linear in general, hence the arguments of the proof of Theorem~\ref{main-BK} may only be carried over in this limited way. The second and more fundamental point is that, as observed in \cite[Lemma~4.5]{BrzRyb:mod}, an $R$-linear heap endomorphism of  $M$ necessarily maps the zero of $M$ into itself. As a consequence the set of all $R$-linear endomorphisms of the heap  $M$ coincides with the set of endomorphisms of $M$ over the ring $R$. The zero map $\hat{0}$ is the absorber of the latter understood as a truss with respect of composition, that is, $\alpha\,\hat{0} = \hat{0}\, \alpha =\hat{0}$ (note that this is not the case in $E(M)$ but only in the $R$-linear part of $E(M)$). Since homomorphisms of trusses preserve absorbers and the retract of a truss at an absorber is a ring, isomorphisms of trusses of endomorphisms of modules over a ring coincide with isomorphisms of corresponding endomorphism rings. And the class of these is too restrictive to capture all isomorphisms between modules.

Finally, let us stress that preceding remarks do not contradict the validity of  Theorem~\ref{main-BK}. Although any abelian group $G$ is a $\Z$-module, by looking at $E(G)$ and treating $G$ as a heap we depart from viewing it not only as a module over the ring $\Z$ but also as a module over the associated truss. This provides one with the required flexibility of the structure to capture all isomorphisms of abelian groups.

In turns out that to implement the Baer-Kaplansky theorem for modules over ring a new concept of the truss of $R$-linear heap endomorphisms is needed. 

Recall from \cite{Brz20} that, given a truss $T$, an abelian heap $M$ together with the associative action $\cdot$ of the multiplicative semigroup of $T$ that distributes over the ternary operations, that is, for all $t, t_1,t_2,t_3\in T$ and $m, m_1,m_2,m_3\in M$,
$$
[t_1,t_2,t_3]\cdot m = [t_1\cdot m,t_2\cdot m,t_3\cdot m], \quad t\cdot[m_1,m_2,m_3] =[t\cdot m_1,t\cdot m_2,t\cdot m_3],
$$
is called a {\em left $T$-module}.  Any element $e\in M$ induces a new $T$-module structure on $M$ with the action
$$
t\overset{e}\cdot m = [t\cdot m, t\cdot e, e], \qquad \mbox{for all $t\in T$ and $m\in M$}.
$$

Every ring $R$ can be viewed as a truss with the same multiplication as that in $R$ and with the (abelian group) heap structure $[r,s,t] = r-s+t$. A left module $(M,+)$ over a ring  $R$ is a left module over the truss $R$: $M$ is viewed as a heap in a natural way and the $R$-action $\cdot$ is unchanged (we will denote it by juxtaposition). Note that in this case the scalar multiplication $\overset{e}\cdot$ is exactly the multiplication induced from the $R$-action on $M$ by imposing the condition that the isomorphism of abelian groups $(M,+)\lra (M,+_e)$, $x\lto [x,0,e]$, is an isomorphism of $R$-modules. In order to see this,  write every element $m\in M$ as $m=[[m,e,0],0,e]$. The above condition is then equivalent to 
$$
r\overset{e}\cdot m=[r[m,e,0],0,e]=[rm,r e,e]\textrm{ for all }m\in M.
$$
Since $0\cdot m =0$, the induced action $\overset{0}\cdot$ is equal to the original action $\cdot$.

Given a left module $M$ over a ring $R$, the family $\CH(M)$ of all $R$-modules $(M,+_e)$ whose scalar multiplications are $\overset{e}\cdot$ is  called \textit{the heap of the module $M$.} If $M,N$ are $R$-modules, a \textit{morphism of heaps of modules},  $\varphi :\CH(M)\lra \CH(N)$,
is a map $\varphi :M\lra N$ that for every $e\in M$ is a homomorphism of $R$-modules $\varphi :(M,+_e)\lra (N,+_{\varphi(e)})$.  
We denote by $H_R(M,N)$ the set of all morphisms of heaps of modules $\CH(M)\lra \CH(N)$.

The proof of the following lemma is straightforward.

\begin{lemma}\label{heap-mod-morphisms}
For a ring $R$  and left $R$-modules $M$ and $N$, 
$$
\begin{aligned}
H_R(M,N) := \big\{\varphi\in \mathrm{Heap}(M,N)\;|\;  \varphi(rm) &= r\varphi(m) -r\varphi(0) +\varphi(0)\\
&=[r\varphi(m),r\varphi(0),\varphi(0)], \quad \forall r\in R, m\in M \big\}.
\end{aligned}
$$
Moreover, 
\begin{itemize}
\item[(i)]
A heap homomorphism (resp.\ isomorphism) $\varphi: M \lra N$ is an element of $H_R(M,N)$ if and only if the map
$$
\widetilde{\varphi}: M\lra M, \qquad m\lto \varphi(m)-\varphi(0),
$$
is a homomorphism (resp.\ isomorphism) of $R$-modules.

\item[(ii)] $\mathrm{Hom}_R(M,N) \subseteq  H_R(M,N)$.

\item[(iii)] The set of endomorphisms of the heap of $M$, $E_R(M): = H_R(M,M)$, is a sub-truss of $E(M)$.
\end{itemize}
\end{lemma}

\begin{remark}
Another way of interpreting the set $\mathrm{H}_R(M,N)$ is provided by observing that the correspondence described in the statements of Lemma~\ref{heap-mod-morphisms} can be lifted to the bijection 
$$
{H}_R(M,N) \overset{\cong}\lra N\times \mathrm{Hom}_R(M,N) , \qquad \varphi \lto \big(\varphi(0), \widetilde\varphi\big).
$$ 
The inverse is given by
$$
N\times \mathrm{Hom}_R(M,N) \lra  {H}_R(M,N), \qquad \big(n,\widetilde\varphi\big)\lto \big[\varphi: m\lto n+\widetilde\varphi(m)\big].
$$
The heap structure induced by this isomorphism is that of the product of heaps, that is,
$$
\big[(n_1,\widetilde\varphi_1), (n_2,\widetilde\varphi_2), (n_3, \widetilde\varphi_3)\big] = \big(n_1-n_2+n_3, \widetilde\varphi_1-\widetilde\varphi_2+\widetilde\varphi_3\big),
$$
for all $n_1,n_2,n_3\in N$ and $\widetilde\varphi_1, \widetilde\varphi_2,  \widetilde\varphi_3\in \mathrm{Hom}_R(M,N)$.
\end{remark}

In the following result we will prove that $E_R(M)$ determines a left  $R$-module $M$ as a module over its endomorphism ring $\End_R(M)$ with the action given by evaluation. If $R$ and $S$ are rings, $M$ is a left $R$-module, and $N$ is a left $S$ module, then we say that the modules $M$ and $N$ are \textit{equivalent} if and only if there exists a group isomorphism $\mu:M\lra N$ and a ring isomorphism $\varrho:R\lra S$ such that $\mu(rm)=\varrho(r)\mu(m)$ for all $m\in M$ and $r\in R$. 

\begin{theorem}[The Baer-Kaplansky theorem for modules]
\label{main-BK-mod} 
Let  $R$ and $S$ be rings, and let $M$ be a left $R$-module and $N$ be a left $S$-module. The trusses $E_R(M)$ and $E_S(N)$ are isomorphic if and only if
$M$ and $N$ are equivalent as modules over their endomorphism rings.      
\end{theorem}
\begin{proof} If $M$ and $N$ are equivalent as modules over their endomorphism rings, there exist an additive map $\mu:M\lra N$ and a ring isomorphism $\varrho:\End_R(M)\lra \End_S(N)$ such that $\mu(um)=\varrho(u)\mu(m)$, for all $m\in M$ and $u\in \End_R(M)$, or, equivalently, $\mu u=\varrho(u)\mu$. We claim that
$$
\Phi:E_R(M)\lra E_S(N), \qquad \alpha\lto \varrho(\alpha-\alpha(0))+\mu(\alpha(0)) = \varrho(\widetilde{\alpha})+\mu(\alpha(0)),
$$
 is an isomorphism of trusses. Indeed, first noting that $\Phi(\alpha)(0) = \mu(\alpha(0))$, and using that $\varrho$ maps endomorphisms of $M$ into endomorphisms  of $N$ one easily finds that $\Phi(\alpha)\in E_S(N)$. To prove that $\Phi$ respects multiplications, it is enough to observe that, for all $\alpha,\beta \in E_R(M)$,
 $
 \widetilde{\alpha}\widetilde{\beta} = \widetilde{\alpha\beta},
 $
 since, for all $m\in M$,
 $$ 
 \begin{aligned}
\widetilde{\alpha}\widetilde{\beta}(m) &= \alpha(\beta(m) -\beta(0)) - \alpha(0)\\
& = \alpha(\beta(m)) - \alpha(\beta(0)) +  \alpha(0) - \alpha(0) = \widetilde{\alpha\beta}(m),
\end{aligned}
$$
where we have used that $\alpha$ is a heap endomorphism. Therefore, using the definition of the equivalence of modules we find,
$$
\begin{aligned}
\Phi (\alpha \beta) &= \varrho(\widetilde{\alpha\beta}) +\mu (\alpha\beta(0)) = \varrho(\widetilde{\alpha})\varrho(\widetilde{\beta}) + \mu(\widetilde{\alpha}\beta(0) + \alpha(0))\\
&= \varrho(\widetilde{\alpha})\varrho(\widetilde{\beta}) + \varrho(\widetilde{\alpha})\mu(\beta(0)) + \mu(\alpha(0))\\
&= \Phi(\alpha) \left(\varrho(\widetilde{\beta}) + \mu(\beta(0))\right) = \Phi (\alpha)\Phi(\beta) ,
\end{aligned}
$$
as required.

Conversely, let $\Phi: E_R(M)\lra E_S(N)$ be an isomorphism of trusses.
As before, if $m\in M$, we denote by $\widehat{m}$ the constant map $x\lto m$ for all $x\in M$.  
Once it is noted that, for all $m\in M$, $\widehat{m}\in E_R(M)$ the arguments of the first part of the proof of Theorem~\ref{main-BK} may be repeated verbatim to associate a heap isomorphism $\varphi: M\lra N$ corresponding to the truss isomorphism $\Phi: E_R(M)\lra E_S(N)$ by the formula 
$$
\widehat{\varphi(m)} = \Phi(\widehat{m}), \qquad \mbox{for all $m\in M$}.
$$
Define  the isomorphism of abelian groups $\mu:M\lra N$ by $\mu=\widetilde{\varphi}$, 
 and  set
$$
\varrho:\End_R(M)\lra E_S(N),\qquad u\lto \Phi(u)-\Phi(u)(0)=[\Phi(u),\widehat{\Phi(u)(0)},\widehat{0}].
$$
Since $\varrho(u)(0)=0$, it follows that $\varrho(u)\in\End_S(N)$ for all $u\in \End_R(M)$.
Furthermore, since
$$\Phi(u+v)=\Phi([u,\widehat{0},v])=[\Phi(u),\Phi(\widehat{0}),\Phi(v)]=\Phi(u)+\Phi(v)-\Phi(\widehat{0}),
$$
it can be proven  that $\varrho$ is additive by direct calculations.

For all $u\in \End_R(M)$ and $m\in M$ we can compute,
\begin{align*}
\varrho(u)\mu(m)&=[\Phi(u)([\varphi(m),\varphi({0}),0]), \Phi(u)(0),0]\\
&= [[\Phi(u)(\varphi(m)),\Phi(u)(\varphi({0})),\Phi(u)(0)], \Phi(u)(0),0] \\
&= [\Phi(u)(\varphi(m)),\Phi(u)(\varphi({0})),0] \\
&= [\Phi(u)(\Phi(\widehat{m})(0)),\Phi(u)(\Phi(\widehat{0})(0)),0] \\
&= [\Phi(u\widehat{m})(0),\Phi(u\widehat{0})(0),0] 
= [\Phi(\widehat{u(m)})(0),\Phi(\widehat{0})(0),0]\\ 
&= [\widehat{\varphi(u(m))}(0),\widehat{\varphi(0)}(0),0]=[\varphi(u(m)),\varphi(0),0] = \mu u(m).
\end{align*}
Hence $\varrho(u) \mu=\mu u$.  In particular, $\varrho(u)=\mu u\mu^{-1}$, which immediately implies that $\varrho$ is an injective multiplicative map. 

Let $v\in \End_S(N)$ and $\alpha\in E_R(M)$ such that $\Phi(\alpha)=v$. Then 
\begin{align*}
\varrho(\widetilde\alpha)&= \Phi(\widetilde\alpha)-\Phi(\widetilde\alpha)(0)=\Phi([\alpha,\widehat{\alpha(0)},\widehat{0}])-\Phi([\alpha,\widehat{\alpha(0)},\widehat{0}])(0)\\
&= [\Phi(\alpha),\Phi(\widehat{\alpha(0)}),\Phi(\widehat{0})]-[\Phi(\alpha),\Phi(\widehat{\alpha(0)}),\Phi(\widehat{0})](0)\\
&=\Phi(\alpha)-\Phi(\alpha)(0)=v-v(0)=v,
\end{align*}
where we have used that $\Phi$ is a heap homomorphism and that it maps constant homomorphisms  to constant ones. Hence $\varrho$ is surjective.
Therefore,  $\varrho$ is a ring isomorphism as required.
\end{proof}

In general, the existence of an isomorphism of trusses $E_R(M)\cong E_R(N)$ does not guarantee that the initial $R$-modules are isomorphic.

\begin{example}\label{non-iso}
Let $F$ be a field and $R=F\times F$. If $M=F\times 0$ and $N=0\times F$, it is easy to check that $E_R(M)\cong E_R(N)$, while there is no $R$-module isomorphism connecting $M$ with $N$. 
\end{example} 

\begin{remark}
As in the case of abelian groups, every truss isomorphism $\Phi:E_R(M)\lra E_S(N)$ is induced uniquely by 
a  heap isomorphism $\varphi\in H(M,N)$ defined in the proof of Theorem \ref{main-BK-mod}, via  $\Phi(\alpha)=\varphi \alpha \varphi ^{-1}$, for all $\alpha\in E_R(M)$. However, not all heap isomorphisms $\varphi :M\lra N$ induce  an isomorphism of trusses $E_R(M)\lra E_S(N)$ in this way since in general $\widetilde{\Phi(\alpha )} = \varphi \alpha \varphi ^{-1}-\varphi \alpha \varphi ^{-1}(0)$ need not be  $S$-linear.  

For instance, let $F=\mathbb{R}$ in Example \ref{non-iso}. Since $\sqrt{2}$, $\sqrt{3}$, and $\sqrt{6}$ are linearly independent over $\mathbb{Q}$, there exists an additive isomorphism $\varphi:\mathbb{R}\lra \mathbb{R}$ such that $\varphi(\sqrt{2})=\sqrt{3}$, $\varphi(\sqrt{3})=\sqrt{6}$, and $\varphi(\sqrt{6})=\sqrt{2}$. If $\alpha:\mathbb{R}\lra \mathbb{R}$, $\alpha(x)=\sqrt{2}\cdot x$, then $\varphi \alpha \varphi^{-1} (\sqrt{2})=\varphi(\sqrt{2}\sqrt{6})=\varphi(2\sqrt{3})=2\sqrt{6}$, and    $\varphi \alpha \varphi^{-1}(\sqrt{6})=\varphi(\sqrt{2}\cdot\sqrt{3})=\sqrt{2}$. Then $\varphi \alpha \varphi^{-1}(\sqrt{3}\cdot\sqrt{2})\neq \sqrt{3}\,\varphi \alpha\varphi^{-1}(\sqrt{2})$, hence $\varphi \alpha \varphi^{-1}$ is not $\mathbb{R}$-linear.
\end{remark}

\section{Conclusions}
In this note we have formulated and proven the Baer-Kaplansky theorem for all abelian groups and all modules over not necessarily commutative rings. It turns out that such a formulation is possible provided we abandon the classical group and ring theory point of view and adopt a more general albeit less standard perspective of  heaps and trusses. This departure from a classical universe of groups and rings is not made out of choice but out of necessity as it  seems impossible to confine the information about isomorphism classes of abelian groups to a structure consisting of a set with two binary operations, one distributing over the other. Trusses that feature in the Baer-Kaplansky theorem may be interpreted as arising from a modification of the notion of a group homomorphism; rather than dealing with group homomorphisms one needs to deal with homomorphisms of associated heaps instead. In a similar way, the Baer-Kaplansky theorem for modules relies on a modification of the notion of a homomorphism of modules.  The latter points to an exciting possibility of developing a new approach to module theory in which a single module over a ring is replaced by a heap of modules (as defined in Section~\ref{sec.mod}). While propelled by the philosophy of exploring fully  ternary  formulation of group axioms, this still goes beyond the study of modules over trusses already undertaken in \cite{Brz20}, \cite{BrzRyb:mod} or \cite{BrzRyb:fun} and constitutes the subject of on-going investigations \cite{BreBrz:hea}.

\end{document}